\documentclass[12pt]{amsart}
\usepackage{amssymb}

\newtheorem{theorem}{Theorem}[section]
\newtheorem{lemma}[theorem]{Lemma}
\newtheorem{proposition}[theorem]{Proposition}

\theoremstyle{remark}
\newtheorem*{definition}{{\bf Definition}}

\theoremstyle{remark}
\newtheorem*{remark}{{\bf Remark}}

\numberwithin{equation}{section}

\newcommand{\CC}{\mathbb{C}}

\newcommand{\dd}{{\rm d}}

\begin{document}

\baselineskip=16pt

\title[Abelian Vortices with Singularities]{Abelian Vortices with Singularities}

\author[J. M. Baptista]{J. M. Baptista}
 
\address{Centre for Mathematical Analysis, Geometry, and Dynamical Systems (CAMGSD), Instituto Superior T\'ecnico, Av. Rovisco Pais, 1049-001 Lisbon, Portugal}

\email{joao.o.baptista@gmail.com}

\author[I. Biswas]{Indranil Biswas}

\address{School of Mathematics, Tata Institute of Fundamental
Research, Homi Bhabha Road, Bombay 400005, India}

\email{indranil@math.tifr.res.in}



\date{}

\begin{abstract}
 Let $L \longrightarrow X$ be a complex line bundle over a compact connected Riemann
surface. We consider the abelian vortex equations on $L$ when the metric on the
surface has finitely many point degeneracies or conical singularities and the
line bundle has parabolic structure. These conditions appear naturally in the study of vortex configurations with constraints, or configurations
invariant under the action of a finite group. We first show that the moduli space of singular vortex solutions is the same as in the regular case. Then we compute the total volume and total scalar curvature
of the moduli space singular vortex solutions. These numbers differ from the case of regular vortices by a very natural term. Finally we exhibit explicit non-trivial vortex solutions over the thrice punctured hyperbolic sphere.
\end{abstract}

\maketitle

\section{Introduction}

Let $X$ be a connected compact Riemann surface and $L \longrightarrow X$ a
complex line bundle of degree $d$. For a given K\"ahler metric $\omega_X$ 
on the 
surface, the abelian vortex equations on $L$ can be formulated in two equivalent ways \cite{Brad}. In the first formulation, one fixes a hermitian metric $h$ on the line bundle and looks for a unitary connection $A$ and a nonzero section $\phi$ that satisfy
\begin{align}
&\bar{\partial}_A \phi \ = \ 0  \label{1.1} \\
& F_A - i e^2 \big(\vert \phi \vert^2_h - \tau \big)\, \omega_X \ =\ 0\, \ ,
\end{align}
where $\tau$ and $e^2$ are positive real constants. In the second formulation, one looks for a holomorphic structure on $L$, a nonzero holomorphic section $\phi$ 
and a hermitian metric $h$ on $L$ such that 
\begin{equation}\label{v-eq}
 F_h - i e^2 \big(\vert \phi \vert^2_h - \tau \big)\, \omega_X \ =\ 0\, ,
\end{equation}
where $F_h$ denotes the curvature of the corresponding
Chern connection. So the condition in \eqref{1.1}
follows from the holomorphicity of $\phi$ in the second formulation,
and the metric $h$ becomes an independent variable in the
second formulation; the Chern connection $A_h$ coincides with the
connection in the first formulation. In this article we will use the
second point of view.

The equations above are the original and simplest example of vortex equations. 
A lot is known about its solutions. One knows, for instance, that for any given 
stable pair $(L, \phi)$, meaning a holomorphic line bundle together with a 
nonzero holomorphic section, if the volume of $X$ is sufficiently large, more
precisely,
\begin{equation}\label{stability}
{\text{\rm Vol}}_\omega \, X \ > \ \frac{2 \pi}{e^2 \tau} \, \deg L \ ,
\end{equation}
then there is exactly one hermitian metric $h$ on $L$ that satisfies 
\eqref{v-eq}. 
This means that the moduli space $M_{vort}^d$ of vortex solutions on a line bundle of degree $d$ is isomorphic 
to space of all isomorphism classes of
stable pairs $(L, \phi)$, which, in turn, is the space of degree 
$d$ effective divisors on $X$, or, equivalently, the symmetric product ${\rm Sym}^d (X)$. If 
condition \eqref{stability} is not satisfied, then there are no solutions of 
\eqref{v-eq} for any stable pair.

Replacing $X$ with higher dimensional K\"ahler manifolds or replacing $L$ with 
more general bundles, the vortex equations and their moduli spaces 
of solutions can be generalized in many different and useful ways (see, for instance, \cite{BDGW, CGMS}). In this article 
we will examine another generalization, distinct from the ones just cited. We 
will consider the vortex equation \eqref{v-eq} on Riemann surfaces with 
singular or degenerate metrics $\omega_X$, and on holomorphic line bundles 
with parabolic structure.

More precisely, we will consider K\"ahler metrics on $X$ that are $C^\infty$ and non-degenerate
on the complement $X\setminus\{q_1\, ,\cdots\, ,q_l\}$ of a finite set of points, and
satisfy the following condition around each singular point $q_j$: there is an open neighborhood $U_j\, 
\subset\, X$ of $q_j$ and a holomorphic coordinate $z\, :\, U_j\, \longrightarrow\, \mathbb C$,
with $z(q_j)\, =\, 0$, such that the restriction of $\omega_X$ to 
$U_j\setminus \{q_j\}$ is of the form
\begin{equation}\label{oU}
\omega_X\vert_{U_j\setminus \{q_j\}}\, =\, f_j(z, \overline{z}) \cdot\vert 
z\vert^{2 \beta_j}
\cdot dz\otimes d\overline{z}\, ,
\end{equation}
where $\beta_j$ is a real number strictly bigger than $-1$ and $f_j$ is a $C^\infty$ function on $U_j$
with $f_j(q_j)\, \not=\, 0$. The above restriction on the weights $\beta_j$ guarantees that
$\omega_X$ has finite total volume. If $\beta_j$ is positive, the metric
is continuous but degenerate at the point. When $\beta_j$ is negative the metric is said to have a conical singularity.

Regarding parabolic singularities on the line bundle, we will consider hermitian metrics $h$
on $L$ that are $C^\infty$ and non-degenerate over the complement $X\setminus\{p_1\, ,
\cdots\, ,p_r\}$ of a finite set of points, and satisfy
the following condition around each point $p_i$: for any holomorphic coordinate function
$z$ around the singularity with $z(p_i)\,=\, 0$,
if $s$ is a holomorphic section of $L$ defined on $U_i$
with $s(p_i)\, \not=\, 0$, then
\begin{equation}\label{par-assympt}
\Vert s\Vert_h^2 \, =\, g_i(z, \overline{z})\cdot \vert z\vert^{2 \alpha_i}
\end{equation}
on $U_i\setminus\{p_i\}$,
where the weight $\alpha_i$ is a non-negative real number, $g_i$ is a continuous function on $U_i$ with $g_i(p_i) \, \not=\, 0$, and $\Vert s\Vert_h$ is the pointwise norm
of the section $s$ with respect to the hermitian structure $h$.
Choosing a holomorphic local trivialization of $L$, the Chern connection $A_h$ of this hermitian metric can be written on an
open neighborhood $U_i$ of the parabolic singularity as
\begin{equation}\label{singular-connection}
A_h\, \vert_{U_i\setminus \{p_i\}} \ = \ \alpha_i \, z^{-1} \dd z \ + 
\eta_i\, ,
\end{equation}
where $\eta_i$ is a $(1,0)$-form on $U_i$. Since
the derivative $\bar{\partial}$ annihilates the first term, the curvature $F_h$ is more regular at the singularity. Using Stokes' theorem it can be shown that, in the presence of parabolic points with sufficiently regular $g_i$s, $i\,\in\, [1\, , r]$, the curvature of the Chern connection has integral
\begin{equation}\label{par-degree}
\frac{\sqrt{-1}}{2 \pi} \int_{X\setminus \{ p_i\}} F_h \ = \ \deg L \ + \ \sum_i \alpha_i \ =: \ \text{\rm par-deg}\, L \ ,
\end{equation}
and this number is usually called the parabolic degree of $L$. The parabolic singularities
on line bundles that we are considering here are a special instance of the much broader
and richer subject of parabolic structures on vector and principal bundles \cite{MY}, \cite{BBN}.

$\ $
One may wonder what is the motivation to study vortices with singularities. On the one hand, 
this article may be regarded as part of an effort to study natural differential objects on 
parabolic bundles and on punctured Riemann surfaces; an effort that so far has been mostly 
centered on the Hermitian--Einstein equation, Higgs bundles and on holomorphic curves. In the 
case of the singular vortex equations important results have been obtained about the 
existence of solutions \cite{BG} and, in a broader setting, about topological properties of the moduli space \cite{MT}, but not yet about metric properties of the moduli space. On the 
other hand, even if one cares only about smooth vortices, parabolic bundles and singular 
metrics appear very naturally as an effective tool to describe vortices that are invariant 
under a given symmetry or that are subject to constraints, as we now explain.

Suppose that a finite group $\Gamma$ acts holomorphically, effectively and
isometrically on a surface $Y$, and that this action lifts to a holomorphic
line bundle $\widehat{L}$ over $Y$. Then $\Gamma$ also acts on sections, connections
and hermitian metrics on $\widehat{L}$. This action takes vortex solutions to vortex solutions and commutes with gauge transformations, so we get
an action of $\Gamma$ on the moduli space $M_{\text{vort}}$ of smooth vortices. The fixed point subvariety $M^\Gamma \subset M_{\text{vort}}$ corresponds to $\Gamma$-invariant vortex configurations and, both from a physical and a mathematical viewpoint, is a natural space to study. One can wonder, for instance, about the geometry of $M^\Gamma$.
It turns out that perhaps the most efficient way to study $\Gamma$-invariant 
vortices on $Y$ is to look at usual vortices on the quotient surface $X = Y / 
\Gamma$. This surface can always be endowed with a complex structure such that 
the projection $\pi : Y \longrightarrow X$ is holomorphic. Crucially, however, 
the quotient metric on $X$ may have singularities. If a point
 $p \in Y$ has a non-trivial isotropy group $\Gamma_p \subset \Gamma$, the quotient metric at $\pi (p)$ will explode like $|z|^{2(1-b)/b}$, where $z$ is a complex coordinate around $\pi (p)$ and the integer $b$ is the cardinality of the subgroup $\Gamma_p$. So we are naturally led to study vortices on surfaces equipped with singular metrics.

A second motivational example is the following. Suppose that one wants to 
study configurations of $m$ smooth vortices on $X$ with the constraint that at 
least $l < m$ vortices are located at a given point $p \in X$. These 
configurations define a subvariety $M_{l.p} \subset M^m_{\text{vort}}$ that is 
biholomorphic, but not isometric, to the smaller moduli space $M^{m - 
l}_{\text{vort}}$. This subvariety is the image of the holomorphic embedding
$M^{m - l}_{\text{vort}} \,\longrightarrow \,M^m_{\text{vort}}$ defined by 
adding $l$ vortices at the location $p$, or, more formally, by the map
\begin{equation}\label{embedding}
(L'\, , \phi' \, , h') \ \longmapsto \ \big( L' \otimes {\mathcal O}_X(p)^{\otimes l} \, , 
\, \phi' . (s_p)^l \, , \, h' . (h_p)^l \big) \ ,
\end{equation}
where ${\mathcal O}_X (p)$ is the holomorphic line bundle over $X$ that has a
section $s_p$ with a single zero located at the point $p$, and $h_p$ is the singular
hermitian metric on ${\mathcal O}_X (p)$ determined by $|s_p|^2_{h_p} = 1$. The inverse of the
map in \eqref{embedding} allows us to identify vortex configurations $(L, \phi, h)$ in
the subvariety $M_{l.p}$ with unconstrained configurations of $(m-l)$ vortices -- a significant simplification. The price to pay for this simplification is that the resulting
unconstrained vortices will now be degenerate at the point $p$. In fact, since $h_p$ explodes
at $p$, the metric $h = h' . (h_p)^l$ will be an honest hermitian metric on
$L\,= \,L' \otimes {\mathcal O}_X(p)^{\otimes l}$ only if the metric $h'$ vanishes at $p$. In this case $h'$ must behave like $|z|^{2l}$ around $p$. Thus constrained vortex configurations can be identified with unconstrained vortices with degeneracies, and we are naturally led to study vortices on line bundles with degenerate hermitian metrics.

An outline of the article is the following. In Section 2 we prove the existence of vortex 
solutions on parabolic line bundles over surfaces with degenerate or singular K\"ahler 
metrics. When these K\"ahler metrics are smooth, the existence of parabolic solutions to the 
(coupled) vortex equations has been established by Biquard and Garc\'{i}a-Prada \cite{BG}.

In Sections 3 and 4 we study the natural $L^2$-metric on the moduli space $M_{\rm vort}$ of vortex solutions. Roughly speaking, if $(\phi_t, h_t )$ is a one-parameter family of vortex solutions, a tangent vector to this path on the moduli space can be represented by the derivative $\dot{\phi}$ of the section and the derivative $\dot{A_h}$ of the Chern connection. The metric on the moduli space is then determined by the $L^2$-norm
\begin{equation} \label{vortexmetric}
\vert \dot{A}_h + \dot{\phi} \vert^2 \ = \ \int_X \Big(\, \frac{1}{4e^2} \: \vert \dot{A_h} \vert^2 \ + \ \vert \dot{\phi} \vert^2_h \, \Big)\, \omega_X
\end{equation}
applied to the horizontal part of the tangent vector, i.e., to the component of $( \dot{A}_h , \dot{\phi})$ perpendicular to all vectors tangent to gauge transformations.
 We show that this metric is well-defined in the case of vortices with singularities and, extending results already established in the smooth case \cite{MN, Perutz, Ba}, compute the total volume and the total scalar curvature of the moduli space for this metric.

In Section 5 we look at the vortex equations for abelian $n$-pairs, also 
called semilocal vortices, in the presence of singularities. Once again we 
start by establishing the existence of vortex solutions, then compute the total volume and scalar curvature of the $L^2$-metric on the moduli space and, finally, as in \cite{Ba}, use those computations to conjecture a formula for the $L^2$-volume of the space of holomorphic curves from $X$ to projective space. Finally in Section 6 we have a brief look at (singular) vortices on hyperbolic surfaces, extend an observation of Manton and Rink \cite{MR}, and exhibit explicit and nontrivial vortex solutions on the thrice punctured hyperbolic sphere.

As a simple and concrete illustration of our results, namely those of Sections 2-4, take the example of $\Gamma \,=\,
\mathbb{Z}_b$ acting on the sphere $Y \,=\, \CC \mathbb{P}^1$ by discrete rotations around the vertical axis. Let $M^m_{l_N , l_S} \subset M^m$ be the moduli space of $\Gamma$-invariant $m$-vortex configurations such that at least $l_N$ vortices are located at the north pole and $l_S$ vortices are located at the south pole. Then these configurations are equivalent to configurations of $m-l_N -l_S$ vortices on the surface $X = \CC \mathbb{P}^1 / \mathbb{Z}_b \simeq \CC \mathbb{P}^1$ equipped with a K\"ahler metric $\omega_X$ with conical singularities at the poles of weight $\beta = (1-b)/b$, and a hermitian metric $h$ with parabolic singularities of weight $\alpha_N = l_N / b$ (respectively, $\alpha_S = l_S / b$) at the north
(respectively, south) pole. Our results for vortices with singularities then say that
$M^m_{l_N , l_S}\,\simeq\, {\text {\rm Sym}}^{m - l_N -l_S} (\CC \mathbb{P}^1)$ and that, with
respect to $L^2$-metric inherited from the moduli space $M^m_{\rm vort}\,=\, {\rm Sym}^{m}(\CC \mathbb{P}^1)$,
the volume of the subvariety is 
\[
{\text {\rm Vol}}\, M^m_{l_N , l_S} \ = \ \frac{\pi^{m - l_N -l_S}}{(m - l_N -l_S)!} \, \Big(\,\frac{\tau}{b} \, 
\Big( {\rm Vol}\, Y \, - \, \frac{2 \pi}{e^2\tau} (l_N + l_S) \Big) \, -\, \frac{2 \pi}{e^2}\, (m - l_N -l_S) \,
\Big)^{m - l_N -l_S}\, .
\]
This coincides with the $L^2$-volume of the moduli space of $m - l_N -l_S$ regular vortices living on a sphere $\CC \mathbb{P}^1$ of Riemannian volume $b^{-1} \, ( {\rm Vol}\, Y - 2 \pi (l_N + l_S)/(e^2 \tau)\, )$. Heuristically, this decrease in the effective volume of $Y$ can be interpreted by saying that each vortex fixed at the poles of $\CC \mathbb{P}^1$ occupies a volume equal to $2 \pi / (e^2 \tau)$, and the remaining volume of the sphere is divided by $b$ due to the imposed $\mathbb{Z}_b$ symmetry.

Beyond this simple example, the results of Sections 2-4 are applicable to more general settings of higher genus surfaces, other groups $\Gamma$, and also to metrics $\omega$ and $h$ whose singularities and degeneracies cannot be obtained as quotients of regular metrics by actions of finite groups.

\section{Existence of vortex solutions}

Let $X$ be a compact connected Riemann surface, and let $\omega$
be a K\"ahler metric on $X$ 
with a finite number of singularities. Let $L \longrightarrow X$ be a 
holomorphic line bundle equipped with a nonzero holomorphic section $\phi$ and 
parabolic weights $\{\alpha_1\, , \cdots\, , \alpha_r\}$ at the parabolic
points $\{p_1\, , \cdots\, , p_r\}$. We assume that the parabolic weights
 satisfy $\alpha_i \geq 0$; in particular we allow $\alpha_i\, \geq\, 1$.
We assume that the singularities of $\omega$ are
represented by a formal sum $\sum_{j=1}^l \beta_j \, q_j$, in the sense of
\eqref{oU}, with weights $\beta_j > -1$. In particular, $(X, \omega)$ has a finite volume. 
In the case $X\,=\, {\mathbb C}{\mathbb P}^1$ we make the additional assumption
\begin{itemize}
\item The set $\{ j \in \{ 1, \ldots, l \} : \beta_j < 0 \}$ has cardinality different from 1.
\end{itemize}

In this section we will show that there exists a hermitian metric $h$ on $L$ 
that satisfies the vortex equation \eqref{v-eq} and is compatible with the 
parabolic data in the sense of \eqref{oU} and \eqref{par-assympt}. This 
extends the classic Hitchin-Kobayashi correspondence for abelian vortices 
\cite{Brad, BG} to our singular setting. As usual, integrating the vortex 
equation \eqref{v-eq} over $X$, one recognizes that a necessary condition for existence of a solution is that
\begin{equation}\label{stability-c}
{\text{\rm Vol}}_\omega \, X \ > \ \frac{2 \pi}{\tau e^2} \Big(\deg L \, + \, \sum_i \alpha_i \Big) \ ,
\end{equation}
at least for solutions such that the Chern formula \eqref{chern-formula} given below
remains valid. Theorem \ref{existence} shows that, just as in the smooth case, once this inequality is satisfied then
there are no more obstacles to the existence of vortex solutions.

For convenience in the statement of the theorem, we will consider the parabolic weights $\{\alpha_1\, , \cdots\, , \alpha_r\}$
as a function $\alpha$ on $X$ defined as follows: $\alpha (p_i) \,=\, \alpha_i$
for all $i\,\in\,\{1\, , \cdots\, ,r\}$ and $\alpha (z)\,=\, 0$ if $z\, \not\in\,
\{p_1\, , \cdots\, ,p_r\}$.
Similarly, we consider the weights of the singular metric $\omega$ as a function
$\beta$ on $X$ defined as $\beta (q_j) \,=\, \beta_j$
for all $j\,\in\,\{1\, , \cdots\, ,l\}$ and $\beta (z)\,=\, 0$ if $z\, \not\in\,
\{q_1\, , \cdots\, ,q_l\}$.

\begin{theorem}\label{existence}
If condition \eqref{stability-c}
is satisfied there exists a unique smooth hermitian metric $h$ on $L$ over
$X\setminus (\{p_1\, , \cdots\, , p_r\}\cup \{q_1\, , \cdots\, , q_l\})$
that satisfies the vortex equation and has the form $h(z) \,=\, e^{2u} |z|^{2\alpha (z_0)}$
on a neighborhood of each singular point $z_0$. The function $u$
is of class $C^{k(\alpha (z_0),\beta(z_0))}$ at the singular point, where the integer
$k(\alpha (z_0),\beta(z_0)) \,\geq\, 0$ is defined in \eqref{definition_k}.
For this solution of the vortex equation, the integral of the curvature of
the Chern connection is
\begin{equation} \label{chern-formula}
\int_X F_h \ =\ - 2\pi \sqrt{-1}\Big(\deg L \, + \, \sum_i \alpha_i \Big)\ .
\end{equation}
\end{theorem}

\begin{definition}
For any two real numbers $\alpha \geq 0$ and $\beta > -1$ we call
\begin{equation}\label{definition_k}
k(\alpha, \beta) \ := \ \begin{cases}
2+ \text{min}\, \big\{\nu(\beta), \ \nu (\alpha + \beta) \big\} &\mbox{if } \ \beta \geq 0 \\
0 &\mbox{if }\ 0> \beta > -1 \ ,
\end{cases} 
\end{equation}
where $\nu(\lambda) = + \infty$ if $\lambda$ is an integer and is the integral part $\nu (\lambda) = [\lambda]$ otherwise. 
\end{definition}

\begin{remark}
Observe that $k(\alpha , \beta) = \infty$ precisely if $\alpha$ and $\beta$ 
are integers. In this case the function $u$ and the curvature $F_h$ are 
$C^\infty$ at the singularity. If $\beta \geq 
0$, then $k(\alpha\, , \beta)\,\geq\, 2$, and so we can still guarantee that $u$ is at 
least $C^2$ and that $F_h$ is at least continuous. For negative $\beta$, the 
curvature $F_h$ has finite flux (it is integrable) but can explode at the 
singularity -- a fact that is apparent directly from the vortex equation.
\end{remark}

\begin{remark}
Theorem \ref{existence} says that once the prescription of singularities is 
fixed, there is a unique vortex solution for each choice of holomorphic 
structure on $L$ together with a nonzero holomorphic section $\phi$. Since the latter 
choices are parametrized by the set of effective divisors on the surface $X$, we
conclude that, just as in the smooth case, the moduli space $M^d_\text{{\rm s-vort}}$
of vortices with singularities is isomorphic to the symmetric product ${\rm Sym}^d (X)$.
\end{remark}

\noindent
The Yang-Mills-Higgs functional of a field configuration $(A, \phi)$ is defined by the integral
\[
E(A, \phi) \ := \ \int_X \mathcal{E}(A, \phi) \ \omega_X
\]
of the standard energy density
\[
 \mathcal{E}(A, \phi) \ := \ \frac{1}{2 e^2} \, \vert F_A \vert^2 \ + \ \vert \dd_A \phi \vert^2 \ + \ \frac{e^2}{2} \big(\vert \phi \vert^2 - \tau \big)^2 \ ,
\]
where the norms are induced by the metric $\omega_X$ on the surface and the hermitian
metric $h$ on the line bundle. When the field configuration $(A_h , \phi)$ is that of a vortex solution with singularities, one can perform the integral over the punctured surface $X \setminus \{ p_i , q_j \}$ to obtain:
\noindent
\begin{proposition}\label{energy}
The vortex solutions of Theorem \ref{existence} have energy
\[
E(A_h, \phi) \ = \ 2 \pi \tau \Big(\deg L \, + \, \sum_i \alpha_i \Big) \ .
\]
\end{proposition}
\noindent
Thus, when singular fields are allowed, the energy of a vortex solution is no longer restricted to be an integral multiple of $2 \pi \tau$. A parabolic point where the hermitian metric $h$ degenerates (so the corresponding Chern connection $A_h$ has a logarithmic singularity, as in
\eqref{singular-connection}) contributes to the total energy with an amount proportional to the parabolic weight. 

\medskip
\noindent
\textbf{Proof of Theorem \ref{existence}.}\, 
The case where all $\beta_j$ are nonnegative is treated in Lemma \ref{le-e} below.
We assume that some $\beta_j$ is negative.

Let $n_j$ be a positive integer such that $\beta_j\,\geq\, 
(1-n_j)/n_j$ and call $\gamma_j \in [0\, , +\infty[$ the difference $\beta_j
- (1-n_j)/n_j$. From \cite[p. 29, Theorem 1.2.15]{Na} we know that there
is a finite (ramified) Galois covering
\begin{equation}\label{f}
f \,:\, Y \,\longrightarrow \,X
\end{equation}
such that each point in the inverse image $f^{-1} (q_j)$ has ramification 
index $n_j -1$. Note that to use \cite[Theorem 1.2.1]{Na}, we need the assumption
made earlier for $X\,=\, {\mathbb C}{\mathbb P}^1$.
Then around each point of $f^{-1}(q_j)$ the map $f$ is like 
$z\,\longmapsto \,
z^{n_j}$, and the pullback of $\omega$ behaves like $f^\ast \omega \, \sim \,
|z|^{2\gamma_j n_j } \, \dd z \wedge \dd \bar{z}$. Notice that by definition 
of $n_j$ we have $\gamma_j n_j \,\geq\, 0$. In particular, the K\"ahler 
form 
$f^\ast \omega$ on $Y$ satisfies the conditions of Lemma \ref{le-e}. So 
consider 
the 
pullback bundle $f^\ast L \,\longrightarrow\, Y$ equipped with the pullback 
section $f^\ast s$. The parabolic weight at the point $f^{-1}(p_i)$ is $\alpha_i$ if
$p_i\, \not\in\, \{q_1\, ,\cdots\, ,q_l\}$, and it is $n_j\alpha_i$ if $p_i\,=\, q_j$. In other
words, the parabolic data on $f^\ast L$ is the formal sum
$\sum_i (r_{f^{-1}(p_i)} +1 )\, \alpha_i \, f^{-1}(p_i)$, where $r_x$ is the ramification index at $x$.
Denoting by $m$ the number of sheets of $f$, we have that
\[
2\pi\, \text{\rm par-deg}\, (f^\ast L) \, - \, \tau e^2 \, {\rm Vol}_{f^\ast \omega}\, Y\, =
\, m \, \big( 2\pi \, \text{\rm par-deg}\, (L) \,- \, \tau e^2\,{\rm Vol}_{\omega}\, X \big) \
\,<\, \ 0\, . 
\]
We can therefore apply Lemma \ref{le-e} to obtain a hermitian metric $h'$ that 
satisfies the vortex equation for $(f^\ast L , f^\ast s )$ on the surface $(Y, 
f^\ast \omega)$ and is compatible with the parabolic data. This hermitian 
metric is continuous everywhere and is $C^\infty$ away from the inverse images 
$f^{-1} (p_i)$ and $f^{-1}(q_j)$.

By uniqueness of vortex solutions on $Y$, 
the metric $h'$ must be invariant under the action of the Galois group of the 
cover $f$, and in particular $h'$ descends to a continuous hermitian metric $h$ 
on the line bundle $L \longrightarrow X$. 
Since $f$ is a local biholomorphism and isometry away from the ramification 
points, $h$ satisfies the vortex equation for $(L, s, \omega)$ and all the 
statements of the theorem are true on the domain $X \setminus \{ q_j \, \mid \, 
\beta_j < 0\}$. 

Now choose a point $q_j \in X$ such that $\beta_j < 0$. As was seen above, at 
each point in $f^{-1} (q_j)$, the pull-back metric $f^\ast \omega$ has a 
singularity with weight $\beta' = \gamma_j n_j$. Since $q_j$ may coincide with 
one of the parabolic points $p_i$, we must also allow for a (possibly zero) 
parabolic weight $\alpha_j$ at $q_j$, so a parabolic weight $\alpha' = n_j 
\alpha_j$ at $f^{-1} (q_j)$. Then Lemma \ref{le-e} asserts that at each point in 
$f^{-1}(q_j)$ the $\text{Gal}(f)$--invariant hermitian metric $h'$ is of the form
$e^{2u'} |z|^{2\alpha'}$, where $u'$ is also invariant and is $C^{k(\alpha', \beta')}$, with
\[
k(\alpha', \beta') \ = \ 2+ \text{min}\, \big\{\nu(n_j \beta_j + n_j -1), \ \nu (n_j\alpha_j + n_j\beta_j +n_j -1) \big\}\ . 
\] 
Since the cover map at the ramification point is like $z \longmapsto w = 
z^{n_j}$, 
the metric $h$ downstairs will be of the form $e^{2u} |w|^{2\alpha_j}$, where $u$, being the quotient of $u'$, is continuous at $q_j$.

The formula for the integral of the Chern curvature $F_h$ follows from the 
integral of $F_{h'}$ as given by Lemma \ref{le-e}:
\begin{align*}
\int_{X} F_h \ &= \ m^{-1} \int_{Y} F_{h'} \ = \ m^{-1} \deg (f^\ast L)\ +\ m^{-1} \deg \Big( \sum_i (r_{f^{-1}(p_i)} +1 )\, \alpha_i \, f^{-1}(p_i) \Big) \\
 &= \ \deg L \ + \ \sum_i \alpha_i \ , 
\end{align*}
where $m$ is the degree of the covering $f$. Finally, uniqueness of the 
vortex solution on $(L, s , \omega)$ follows from the uniqueness on $(f^\ast L, 
f^\ast s, f^\ast \omega)$.$\hfill{\Box}$

\begin{lemma}\label{le-e}
Theorem \ref{existence} is true when the weights $\beta_j$ are non-negative.
\end{lemma}

\begin{proof}
Let $h_0$ be a background hermitian metric on $L$ representing the parabolic data $\sum \alpha_i \, p_i$. Any other hermitian metric with the same parabolic data can be written as $h= e^{2u} h_0$ for some real function $u$ on $X$. Let $\omega_0$ be a background K\"ahler metric on $X$ with no singularities or degeneracies and the same volume as $\omega$. One can write $\omega \,=\, \rho \, \omega_0$, where $\rho$ is a positive function representing the divisor $\sum \beta_j \, q_j$ in the sense of \eqref{oU}. As originally observed by Bradlow in \cite{Brad}, the vortex equation \eqref{v-eq} for $h$ is then equivalent to
\begin{equation} \label{kw-eq}
\Delta_0 u \ = \ K e^{2u} \ - \ K_1 \ ,
\end{equation}
where $\Delta_0$ is the positive-definite Laplacian of the metric $\omega_0$ 
and the coefficients are $$K \,=\, - e^2\, \rho \,|s|^2_{h_0} ~\, ~ \text{ and } ~\,~ K_1 \,= \, \sqrt{-1} \cdot 
\Lambda_0 F_{h_0} - e^2 \tau \, \rho\, .$$
Here $\Lambda_0 F_{h_0}$ denotes the contraction of the curvature $ F_{h_0}$ with the metric $\omega_0$. The functions $K$ and $K_1$ are clearly $C^\infty$ away from the
points $\{q_j\}$ and \{$p_i\}$. At those points they are continuous, since the weights $\beta_j$ and $\alpha_j$ are non-negative. Moreover, these coefficient functions satisfy
\begin{align*}
&(i)\ K \leq 0 \ \ {\text{\rm and is strictly negative on a non-empty domain,}} \\
&(ii) \int_X K_1 \, \omega_0 \ = \ 2 \pi \Big(\deg L + \sum_i \alpha_i \Big) - e^2 \, \tau\, {\rm Vol}_\omega X \ < \ 0 \ .
\end{align*}
It then follows from the classical results of Kazdan and Warner (see in 
particular Corollary 10.13 and the proof leading to Theorem 10.5 in \cite{KW}) 
that \eqref{kw-eq} has a unique solution $u \in C^2 (X)$, and that this 
solution is $C^\infty$ on the domain $X \setminus \{\{q_j\} , \{p_i\}\}$. Moreover, if the functions $K$ and $K_1$ are $C^{k}$ at a given point, then $u$ will be $C^{2+ k}$ at that point. 
Thus formula \eqref{definition_k} in the case $\beta\,\geq\, 0$ follows from the
straightforward observation that if $\beta$ and $\alpha$ are the weights of $\rho$ and $h_0$, respectively, at a given singularity, then $\Lambda_0 F_{h_0}$ is $C^\infty$, while $\rho$ is $C^{ \nu (\beta)}$ and $\rho \,|s|^2_{h_0}$ is $C^{\nu(\alpha + \beta)}$ at the singularity. Finally, using that $h_0$ is a parabolic metric in the sense \eqref{par-assympt} with smooth coefficients $g_i$, and that $u$ is at least $C^2$ on the whole $X$, we can apply Stokes' theorem to obtain 
 \[
\int_X F_h \ = \ \int_X F_{h_0} - \partial \bar{\partial}u \ = \ \int_X 
F_{h_0} \ = \ - 2\pi i \, \big( \deg L + \sum_j \alpha_j \big) \ , 
\]
completing the proof. 
\end{proof}

\noindent
\textbf{Proof of Proposition \ref{energy}.}\, 
We start by proving the result in the case where the metric on the Riemann surface does not explode, meaning in the case where the weights satisfy $\beta_j \geq 0$. The standard Bogomolny argument for the Yang-Mills-Higgs functional \cite{Brad} allows us to rewrite 
\[
 \mathcal{E}(A, \phi) \ = \ \frac{1}{2 e^2} \, \Big\vert F_A - \sqrt{-1} e^2
\big(\vert \phi \vert^2 - \tau \big)\, \omega_X \Big\vert^2 \ + \
2 \big\vert \bar{\partial}_A \phi \big\vert^2 \ + \ \sqrt{-1} \tau F_A \ - \ \sqrt{-1} \, \dd \big(h\, \bar{\phi} \, \dd_A \phi \big) \ .
\]
When the vortex equations are satisfied, the first two terms vanish. By
Theorem \ref{existence}, the integral over $X$ of the third term is $2 \pi \tau \, \big( \deg L + \sum_j \alpha_j \big)$. So we only need to show that the integral over $X$ of the last term vanishes. In the smooth case this is standard and obviously true, by Stokes' theorem. When the connection $A_h$ has logarithmic singularities at the parabolic points, we can still localize the surface integral to a sum of line integrals over small circles $C^j_R$ around the parabolic points,
\[
\int_X \dd \big( h\, \bar{\phi} \, \dd_A \phi \big) \ = \ \sum_{p_j}\ \lim_{R \rightarrow 0}\ \int_{C^j_R} h\, \bar{\phi} \, \dd_A \phi\ .
\]
To evaluate the circle integrals, around each parabolic point take the local holomorphic trivialization of $L$ defined by $\phi \,=\, 1$. Since $A_h$ is the Chern connection in this trivialization,
we have $A_h \,=\, h^{-1} \partial h$, and so $h\, \bar{\phi} \, \dd_A \phi \,=\, \partial(h \bar{\phi} \phi)\, =\, \partial h$. It follows that, in this trivialization,
\begin{align*}
\int_{C_R} h\, \bar{\phi} \, \dd_A \phi \ &= \ \int_{0}^{2\pi} \Big( \frac{\sqrt{-1}R}{2} \,
\partial_R h \ + \frac{1}{2} \,\partial_\theta h\Big) \ \dd \theta \ = \ \frac{\sqrt{-1}R}{2} \, \int_{0}^{2\pi} \partial_R h \ \dd \theta\\
&= \ \sqrt{-1} \, \int_{0}^{2\pi} \Big( R^{2\alpha +1}\, e^{2u}\, \partial_R u \ + \ \alpha\, R^{2 \alpha} \, e^{2u} \Big)\ \dd \theta \ ;
\end{align*}
here we have used that near each parabolic point the hermitian metric can be written in the form $h(z) \,=\, e^{2u} |z|^{2\alpha}$, with $\alpha \geq 0$. But Theorem \ref{existence} says that when the weights $\beta_j$ are non-negative, the function $u$ is at least of class $C^2$, and in particular both $e^{2u}$ and $\partial_R u$ are continuous in a neighborhood of the parabolic point. Thus clearly
\[
\lim_{R \rightarrow 0}\ \int_{C_R} h\, \bar{\phi} \, \dd_A \phi \ = \ 0 \ ,
\] 
completing the proof in the case $\beta_j \,\geq\, 0$.

The proof of the proposition in the general case $\beta_j\, >\, -1$ can be reduced to the case
proved above by the method employed in the proof of Theorem \ref{existence}. Namely, we take the Galois covering map $f\,:\, Y\, \longrightarrow \,X$
as in that proof and consider the vortex equations on the line bundle $f^\ast L$ over the Riemann surface with hermitian structure $(Y, f^\ast \omega_X)$ equipped with the holomorphic section $f^\ast \phi$. By construction all the singular points of the metric $f^\ast \omega_X$ have non-negative weights, so from the first part of the proof of the proposition we know that all vortex solutions on $(f^\ast L, Y, f^\ast \omega_X)$ have energy 
\[
2 \pi \tau \, \text{par-deg} \,( f^\ast L) \ = \ 2 \pi \tau\, m \, \text{par-deg}\, (L)\, ,
\]
where $m$ is the number of sheets of the cover (the degree of $f$). At the same time, from the proof of Theorem \ref{existence} we know that the vortex solution for $(f^\ast L , f^\ast \phi, Y, f^\ast \omega_X)$ is invariant under the action of the Galois group $\text{Gal}(f)$, that it descends to the vortex solution for $( L, \phi, X, \omega_X)$, and that all the vortex solutions downstairs on $X$ can be obtained in this way. So for any vortex solution on $X$ we have
\begin{align*}
E _X (A , \phi) \ &= \ \int_X \mathcal{E}(A, \phi) \ \omega_X \ = \ m^{-1}\, \int_Y f^\ast \, \mathcal{E}( A, \phi) \ f^\ast \omega_X \\
&= \ m^{-1}\, E _Y (f^\ast A , f^\ast \phi)\ = \ 2 \pi \tau \, \text{par-deg}\, (L) \, ,
\end{align*}
as desired.
$\hfill{\Box}$

\section{K\"ahler metric on moduli space of singular abelian vortices}

In this section we will check that the $L^2$ metric on the vortex moduli space is well defined even when singularities are allowed, that is, when we allow the hermitian metric on the line bundle to be degenerate at the parabolic points and the K\"ahler metric on the surface $X$ to have point degeneracies or conical singularities.

Take a parabolic abelian pair $\underline{y}\, :=\, (L\, 
,s)\,\in\,
\text{Sym}^d(X)$. The hermitian structure on $L$ given by
Theorem \ref{existence} for the pair $(L\, ,s)$ will be denoted 
by $h_{\underline{y}}$. 
The Chern connection on $L$ over $X'\, :=\, X\setminus\{p_1\, ,
\cdots\, ,p_r\}$ for $h_{\underline{y}}$ will be denoted by 
$\nabla^{\underline{y}}$. For any $\epsilon\,> \,0$, define
$$
D_\epsilon\, :=\,\{t\,\in\, {\mathbb C}\,\mid\,~ |t|\,< \, \epsilon\}\, .
$$
Take a holomorphic family of parabolic 
abelian pairs parametrized by $D_\epsilon$, where $\epsilon$ is
sufficiently small, that deforms the given pair $(L\, ,s)$. This
means that we have a holomorphic line bundle ${\mathcal L}$
over $X\times D_\epsilon$, together with a holomorphic section 
$S$ of ${\mathcal L}$, such that ${\mathcal 
L}\vert_{X\times \{0\}} \,=\, L$ and $S\vert_{X\times \{0\}}
\,=\, s$. Let
\begin{equation}\label{v}
v\, \in\, T_{\underline{y}} \text{Sym}^d(X)
\end{equation}
 be the holomorphic tangent vector given by
this family. For any $t\, \in\, D_\epsilon$, let $h_t$ be the 
hermitian structure on ${\mathcal L}\vert_{X'\times \{t\}}$ 
given by Theorem \ref{existence} for the pair $({\mathcal 
L}\vert_{X\times \{t\}}\, ,S\vert_{X\times \{t\}})$.

The hermitian structures $h_t$, $t\,\in\, D_{\epsilon}$, together define 
a hermitian structure $h$ on ${\mathcal L}\vert_{X'\times D_\epsilon}$.
Consider the Chern connection on ${\mathcal 
L}\vert_{X'\times D_\epsilon}$ for this hermitian structure $h$.
For any point $x\, \in\, X'$, taking parallel translations of 
the fiber ${\mathcal L}_{(x,0)}$ along the radial lines in the
disk $\{x\}\times D_\epsilon$, we get an identification of 
${\mathcal L}_{(x,0)}$
with ${\mathcal L}_{(x,t)}$ for each $t\, \in\, D_\epsilon$. In this
way, a $C^\infty$ isomorphism between the hermitian line bundles
${\mathcal L}\vert_{X'\times \{t\}}$ and
${\mathcal L}\vert_{X'\times \{0\}}$ is obtained.

The Chern connection on ${\mathcal L}\vert_{X'\times \{t\}}$ 
for $h_t$ will be denoted by $\nabla^t$. So,
$\nabla^0\,=\, \nabla^{\underline{y}}$. Using the above 
identification between ${\mathcal L}\vert_{X'\times \{t\}}$ 
and ${\mathcal L}\vert_{X'\times \{0\}}$, we have
\begin{equation}\label{le}
\nabla^t\,=\, \nabla^{\underline{y}} +\theta_t~\text{~and~}~ 
S\vert_{X\times \{t\}}\,=\, s +\sigma_t\, ,
\end{equation}
where $\theta_t$ is a smooth $1$--form on $X'$ 
and $\sigma_t$ is a smooth section of ${\mathcal 
L}\vert_{X'\times \{0\}}\,=\, L$ over $X'$.

Then define the norm
\begin{equation}\label{kf}
\Vert v\Vert^2\, :=\, \frac{1}{2} \int_X 
(\frac{d\theta_t}{dt}(0))\wedge \star (\frac{d\theta_t}{dt}(0)) + 
\int_X \Vert \frac{d\sigma_t}{dt}(0)\Vert^2\cdot\Omega_X\, ,
\end{equation}
where $\theta_t$ and $\sigma_t$ are constructed in \eqref{le},
and $v$ is the tangent vector in \eqref{v}.
To complete the definition, we need to show that the integrals in 
\eqref{kf} are finite. Note that both the $(1\, ,1)$--forms 
on $X'$ that are being integrated in \eqref{kf} are positive.

\begin{lemma}\label{lem4}
Both the $(1\, ,1)$--forms $(\frac{d\theta_t}{dt}(0))\wedge \star 
(\frac{d\theta_t}{dt}(0)) $ and $\Vert \frac{d\sigma_t}{dt}(0) 
\Vert^2\cdot\Omega_X$ over $X'$ in \eqref{kf} have a finite integral.
\end{lemma}

\begin{proof}
Take any parabolic abelian pair $\underline{y}\, :=\, (L\, ,s)$ 
over $X$. As before, let $v\, \in\, T_{\underline{y}}
\text{Sym}^d(X)$ be the holomorphic tangent vector given by a
holomorphic family of parabolic abelian pairs $({\mathcal L}\, 
,S)$ parametrized by $D_\epsilon$. Let $\theta_t$ and
$\sigma_t$ be as in \eqref{le}.

Consider the covering $Y$ in \eqref{f}. The abelian pair on $Y$
corresponding to $$({\mathcal L}\vert_{X\times \{t\}}\, 
,S\vert_{X\times \{t\}})$$ will be denoted by 
$(\widehat{L}_t\, ,\sigma_t)$. Let $\widehat{h}_t$ be the
hermitian structure on $\widehat{L}_t$ solving the vortex
equation for the pair $(\widehat{L}_t\, ,\sigma_t)$ with respect
to the K\"ahler form $\omega_Y$. Let
$\widehat{\nabla}^t$ be the Chern connection for
$\widehat{h}_t$. Write as before,
\begin{equation}\label{fi2}
\widehat{\nabla}^t\,=\, \widehat{\nabla}^0 
+\widehat{\theta}_t ~\text{~and~}~ \sigma_t\,=\, 
\sigma_0 + \widehat{\eta}_t\, .
\end{equation}

Recall that the K\"ahler form $\omega_Y$ on $Y$ is the extension 
of the pullback $(f^*\omega_X)\vert_{f^{-1}(X')}$, where $f$ is the 
covering map in \eqref{f}. Also recall that the section $\sigma_t$ is 
the pullback of $S\vert_{X\times \{t\}}$. The pullback, over $f^{-1}(X')$, 
of the hermitian structure $h_t$ on ${\mathcal L}\vert_{X'\times 
\{t\}}$ extends to a hermitian structure on 
$\widehat{L}_t$ over $Y$, and this extension satisfies the 
vortex equation for $(\widehat{L}_t\, ,\sigma_t)$. Therefore, 
we conclude that
\begin{equation}\label{fi}
\widehat{\theta}_t\, =\, f^*\theta_t ~\text{~and~}~ 
\widehat{\eta}_t\,=\, f^*\sigma_t
\end{equation}
over $X'$, where $\widehat{\theta}_t$ and $\widehat{\eta}_t$
(respectively, $\theta_t$ and $\sigma_t$)
are constructed in \eqref{fi2} (respectively, \eqref{le}).
From \eqref{fi} we conclude that
\begin{equation}\label{fi3}
\frac{1}{2} \int_Y(\frac{d\widehat{\theta}_t}{dt}(0))
\wedge \star (\frac{d\widehat{\theta_t}}{dt}(0)) + \int_Y \Vert 
\frac{d\widehat{\eta}_t}{dt}(0)\Vert^2\cdot\Omega_Y
\end{equation}
$$
=\, 
\frac{1}{\#{\rm Gal}(f)}(\frac{1}{2} \int_X
(\frac{d\theta_t}{dt}(0))\wedge \star (\frac{d\theta_t}{dt}(0))
+ \int_X \Vert \frac{d\sigma_t}{dt}(0)\Vert^2\cdot\Omega_X)\, .
$$
The right--hand side of \eqref{fi3} is a finite number because the 
left--hand side of \eqref{fi3} is a finite number. Since the 
right--hand side of 
\eqref{fi3} coincides with the right--hand side of \eqref{kf}, 
the lemma follows.
\end{proof}

\section{K\"ahler class, volume and total scalar curvature}

Let $Z$ be a compact connected Riemann surface equipped with
a smooth and non-degenerate K\"ahler metric.
The cohomology class of the K\"ahler form $\omega_{\rm vort}$ of the $L^2$--metric 
on the vortex moduli space $M={\rm Sym}^d(Z)$ is explicitly known 
\cite{MN, Perutz, Ba}. It can be written as a linear combination
\begin{equation}\label{4.1}
[ \omega_{\rm vort}] \ = \ \pi\ \Big(\tau \, {\rm Vol}(Z) - \frac{2 
\pi}{e^2} d \Big) \, \eta \ + \ \frac{2 \pi^2}{e^2} \, \theta 
\end{equation}
of two natural cohomology classes $\eta$ and $\theta$ in $H^2 ({\rm 
Sym}^d (Z),\, {\mathbb R})$; their definitions are recalled below. In this 
section we will see that when the vortices are defined on parabolic line 
bundles, instead of regular line bundles, or when the metric on the Riemann
surface has conical singularities, the expression for the K\"ahler class is a
suitable modification of \eqref{4.1}.

\begin{theorem}\label{prop3}
Let $X$ be a compact Riemann surface equipped with a K\"ahler metric $\omega$ 
with a finite number of singularities, as in \eqref{oU}, and assume that the 
weights at these singularities satisfy $\beta > -1$. Let $L \longrightarrow X$ 
be a holomorphic line bundle equipped with parabolic data $\sum \alpha_i \, 
p_i$, 
where $\alpha_i \geq 0$. Then the K\"ahler class of the natural $L^2$-metric 
on the moduli space ${\rm Sym}^d(X)$ of $d$ vortices on $L \longrightarrow (X, 
\omega)$ is
\begin{equation}\label{1-v-Kclass}
[ \omega^d_{\text{\rm s-vort}}] \ = \ \pi\ \Big(\tau \, {\rm 
Vol}\, X
- \frac{2 \pi}{e^2}\, (d + \sum^r_{i=1} \alpha_i ) \Big) \, \eta \ + \ \frac{2 \pi^2}{e^2} \, 
\theta\, .
\end{equation}
\end{theorem}

\noindent
Thus the only difference with the nonsingular case considered
in \cite{Ba} is that the usual degree is replaced 
by the parabolic degree. Equivalently, the volume of the surface 
$X$ is replaced by ${\rm Vol}\, X - \sum_{i} 2 \pi \alpha_i / (e^2 
\tau)$. In particular, an intersection calculation on $H^\ast 
({\rm Sym}^d (X), \, {\mathbb R})$ similar to the one performed 
in \cite{MN, Ba} for regular vortices gives the total volume and the 
total scalar curvature of the moduli space of parabolic vortices.
These numbers are the following:
\begin{align*}
{\rm Vol}\, M_{\text{s-vort}}^d \ &= \ \int_{{\rm Sym}^d (X)} 
\ \omega_{\text{\rm s-vort}}^d \ / \ d! \\
&=\ \pi^d \, \sum_{i=0}^{{\rm min}\{ g, d \}} \, \frac{g!\: }{i! \: (d-i)! \: (g-i)!} \ (4\pi)^i \, \Big(\tau \, 
\Big( {\rm Vol}\, X \, - \, \frac{2 \pi}{e^2 \tau} \sum_{j=1}^r \alpha_j \Big) \, -\, \frac{2 \pi}{e^2}\, d 
\Big)^{d-i}\, \nonumber
\end{align*}
for the volume, and for the total scalar curvature is
\begin{align*}
& \int_{{\rm Sym}^d (X)} s(\omega_{\text{s-vort}}) \ \omega_{\text{s-vort}}^d \, / \, d! \ = \\
&=\ (2\pi)^d \, \sum_{i=0}^{{\rm min}\{ g, d \}} \ \frac{g!\: \big( d+1 - 2g +i \big)}{i! \: (d -1-i)! \: (g -i)!} \ 
(2\pi)^i \ \Big[ \frac{\tau}{2}\Big( {\rm Vol}\, X \, - \, \frac{2 \pi}{e^2\tau}\, \sum_{j=1}^r \alpha_j\Big) -
\frac{\pi}{e^2}\, d \Big]^{d-1-i}\, \nonumber
\end{align*}
Observe that these formulae reduce to the usual results for regular vortices when the weights $\alpha_i$ of the parabolic singularities vanish.

\begin{remark}
One usually thinks of abelian vortices as 
finite-size objects, since the energy density of the vortex 
solutions is concentrated on ``effective disks" of area $2
\pi /(e^2\tau)$ centered around the zeros of the section 
$\phi$. This physical image is supported by the term ${\rm Vol}\, X - 2\pi d /(e^2\tau) $ that appears in the Bradlow condition \eqref{stability} and in the usual formulae for the volume of the moduli space of $d$ vortices. In the parabolic case, the corresponding term is ${\rm Vol}\, X \,-\, 2\pi (d + \sum_i \alpha_i)/(e^2 \tau)$, so we are lead to the heuristic interpretation that each parabolic singularity occupies a disk on the surface $X$ of effective area $2\pi \alpha_i / (e^2 \tau)$.
To study the statistical-mechanical properties of a ``gas" of 
vortices \cite{MN}, the relevant quantity is the asymptotic behavior of ${\rm 
Vol}\, M_{\text{vort}}^d (X)$ in the thermodynamical limit 
where $d\rightarrow \infty$ and ${\rm Vol} (X) \rightarrow 
\infty$ with constant density $d / {\rm 
Vol}\, X$. When singularities are present, everything works as in the 
regular case with the volume of $X$ reduced by $\sum_{i=1}^r 
2 \pi \alpha_i /(e^2 \tau)$.
\end{remark}

\begin{remark}
The localization technique in topological field theory developed by Moore, Nekrasov and Shatashvili in \cite{MNS} has been recently applied to evaluate and predict volumes of vortex moduli spaces \cite{MOS}. It would be interesting to see if those techniques can be used in the case of parabolic vortices and punctured Riemann surfaces.
\end{remark}

Before proving Theorem \ref{prop3}, we recall the definition of the 
cohomology classes $\theta$ 
and $\eta$ lying in $H^2 ({\rm Sym}^d (X), \mathbb{R})$. 
The class denoted by $\eta$ is
the Poincar\'e dual of the image of ${\rm Sym}^{d-1}(X)$
in ${\rm Sym}^d(X)$ by the embedding that sends any $\{x_1\, ,\cdots\, 
,x_{d-1}\}$ to $\{x_0\, ,x_1\, ,\cdots\, ,x_{d-1}\}$, where
$x_0\, \in\, X$ is a fixed point (for convenience, we take
${\rm Sym}^0(X)$ to be a point); so the
image of ${\rm Sym}^{d-1}(X)$ is the subvariety of ${\rm Sym}^d(X)$
parametrizing all $d$--tuples that contain the fixed point $x_0$.
This integral cohomology class $\eta$ is clearly independent of the choice of
the base point $x_0$.

To define the cohomology class
$\theta$, let $\text{Pic}^d(X)$ be the component of 
the Picard group of $X$ parametrizing all isomorphism classes of 
holomorphic line bundles over $X$ of degree $d$. The variety 
$\text{Pic}^d(X)$ is isomorphic to the Jacobian of $X$. 
The cohomology group $H^2(\text{Pic}^d(X), \, {\mathbb R})$
is canonically identified with $\bigwedge^2 H^1(X,\, {\mathbb 
R})$. The anti--symmetric pairing on
$H^1(X, \, {\mathbb R})$ defined by
$$
\alpha\, ,\beta\,\longmapsto\,\int_X \alpha\wedge\beta
$$
defines a homomorphism $H^1(X, \, {\mathbb R})\,\longrightarrow\,
H^1(X, \, {\mathbb R})^* \,=\, H_1(X, \, {\mathbb R})$. This homomorphism
is an isomorphism because the above pairing is nondegenerate. The inverse
homomorphism $H^1(X, \, {\mathbb R})^*\,\longrightarrow\, H^1(X, \, {\mathbb R})$
produces an element of $\bigwedge^2 H^1(X, \, {\mathbb R})$. Let
$$
\theta_{\text{Pic}^d(X)}\, \in\,
H^2(\text{Pic}^d(X),\, {\mathbb R})
$$
be the element corresponding to it by the
above identification between $\bigwedge^2 H^1(X, \, {\mathbb 
R})$ and $H^2(\text{Pic}^d(X), \, {\mathbb R})$. This element is 
usually called the theta class of ${\text{Pic}^d(X)}$. Now let
\begin{equation}\label{gxd}
\gamma^{}_{X,d}\, :\, {\rm Sym}^d(X)\, 
\longrightarrow\,\text{Pic}^d(X)
\end{equation}
be the morphism that sends any $\{x_1\, ,\cdots\, ,x_d\}$ to
the holomorphic line bundle ${\mathcal O}_X(\sum_{i=1}^d x_i)$.
For convenience, we define define $\gamma^{}_{X,0}({\rm 
Sym}^0(X))$ to be the
trivial line bundle ${\mathcal O}_X$. Then by definition
\begin{equation}\label{dth}
\theta\, :=\, \gamma^*_{X,d}\, \theta_{\text{Pic}^d(X)}\,\in\,
H^2({\rm Sym}^d(X),\, {\mathbb R})\, . 
\end{equation}
Theorem \ref{prop3} will be proved by establishing a sequence of lemmas.

\begin{lemma}
Theorem \ref{prop3} is true when the K\"ahler form $\omega$ on $X$ is smooth, 
i.e., when the weights at the singularities are $\beta_j = 0$.
\end{lemma}

\begin{proof}
In this setting the only singularities present are the parabolic singularities of the hermitian metric on the line bundle.
Recall from \cite{BG} and Section 2 that, as a complex manifold, the moduli space $M^d_{\text{s-vort}}$ vortices with parabolic singularities is isomorphic to ${\rm Sym}^d (X)$ -- the moduli space of the usual and regular vortices.
Using exactly the arguments as in \cite{Perutz, BS, Ba} for the case of 
regular vortices, one can construct a universal line bundle $\mathcal{L} 
\longrightarrow M^d_{\text{s-vort}} \times X$ for vortices with parabolic 
singularities. In fact, as a holomorphic bundle, $\mathcal{L}$ is
the same whether the solutions we consider have parabolic points or not. This
universal bundle also comes equipped with a natural holomorphic section
$\mathcal{S}$ and a natural hermitian metric $H$ with the property that, for
any equivalence class of vortex solutions $q \,=\, [L, s, h]\,\in\, M^d_{\text{s-vort}}$, the
restriction of $(\mathcal{L}, \mathcal{S}, H)$ to the one-dimensional complex
submanifold $\{ q \} \times X \simeq X
\, \subset\, M^d_{\text{s-vort}}\times X$ coincides with $(L, s, h)$. Since the hermitian metric $h$ has degeneracies at the parabolic points, so does $H$. As a function on the product $M^d_{\text{s-vort}} \times X$, we have that:
\[
|S|^2_H \, ([L, s, h], z) \ = \ |s|^2_h \, (z) \ .
\]
Calculations similar the ones performed in \cite{Perutz, BS, Ba} show that the K\"ahler form of the natural $L^2$-metric on the moduli space $M^d_{\text{s-vort}}$ satisfies
\begin{equation}\label{metric_formula}
\omega_{\text{s-vort}} \ = \ \int_X \frac{\sqrt{-1} \tau}{2} \: F_{H} \wedge\omega \ + \ \frac{1}{4 e^2}\: \: \, F_H \wedge F_H \ ,
\end{equation}
where $F_H$ stands for the curvature of the Chern connection on $(\mathcal{L}, H)$. In the parabolic case, however, the metric $H$ has degeneracies, so the Chern connection will not be smooth everywhere. In particular, the curvature $F_H$, even though continuous, does not necessarily represent the Chern class of $\mathcal{L}$.

Now let $\rho$ be a fixed smooth and non-negative function on $X$ that represents
the real divisor $\sum \alpha_i \, p_i$. Then $H_0 := \rho^{-1}
\cdot H$ is a hermitian metric on $\mathcal{L}$ with no degeneracies and
it is at least $C^2$. So 
\begin{equation}\label{univ_curvature}
[F_H] \ = \ [F_{H_0} + \bar{\partial} \partial \log (\rho)] \ = \ 
-2\pi\sqrt{-1}\cdot c_1 (\mathcal{L}) - [\partial \bar{\partial} \log(\rho)]\,\in\,
H^2(M^d_{\text{s-vort}}\times X,\, {\mathbb R})
\end{equation}
as a cohomology class of the base $M^d_{\text{s-vort}} \times X$. Since the function $\rho$ does not depend on the coordinates on the moduli space, the last term is proportional to the fundamental class of $X$. The proportionality constant is
\[
\int_X \partial \bar{\partial} \log(\rho) \ = \ \sum_j\, \lim_{R \rightarrow 0} 
\int_{\partial D_j^R} \partial_{\bar{z}} \log(\rho)\, \dd \bar{z} \ = \ \sum_j 
2\pi\sqrt{-1} \cdot \alpha_j \ ,
\] 
where we have used Stokes' theorem to express the surface integral as a sum of loop integrals around the boundaries of small disks $D_j^R$ of radius $R$ centered at the singular points $p_j \in X$. Then it follows from \eqref{metric_formula} and \eqref{univ_curvature} that
\begin{equation*}
[ \omega_{\text{\rm s-vort}}] \ = \ \int_X \pi \Big( \tau - \frac{2 \pi \sum_j \alpha_j}{e^2\, {\rm Vol}_\omega\, X} \Big)\, c_1 (\mathcal{L}) \wedge\omega \ - \ \frac{\pi^2}{e^2} \,
c_1 (\mathcal{L}) \wedge c_1 (\mathcal{L}) \ .
\end{equation*}
Thus the class $[\omega_{\text{\rm s-vort}}]$ is given by the same formula as the class $[\omega_{\text{\rm vort}}]$, except that one should make the substitution 
\[
\tau\ \longmapsto \ \tau - \sum_j \frac{2 \pi \, \alpha_j}{e^2\, {\rm Vol}_\omega\, X } \ .
\]
Comparing with \eqref{4.1}, this proves the lemma.
\end{proof}

\begin{lemma} \label{lemma2}
Formula \eqref{1-v-Kclass} is still valid when the K\"ahler form $\omega$ on $X$ has degeneracies represented by a real divisor $\sum_j \beta_j q_j$, with $\beta_j \geq 0$.
\end{lemma}
\begin{proof}
Going through the arguments in \cite{Ba}, one can check that
formula \eqref{metric_formula} for the K\"ahler form on the moduli
space ${\text{\rm Sym}}^d (X)$ still holds when $\omega$ has point
degeneracies. Then the proof of the previous lemma can be applied without
any modifications.
\end{proof}

\medskip
\noindent
To complete the proof of Theorem \ref{prop3},
we will reduce the general case $\beta_j > -1$ to the case $\beta_j \geq 0$ treated
above. This will be done using the method used in the proof of Theorem \ref{existence}. 
 
Let $f: Y \longrightarrow X$ be the $m$-sheeted Galois cover described in the proof of Theorem \ref{existence}. The pull-back operation $h \longmapsto f^\ast h$ defines a one-to-one correspondence between vortex solutions on $(L, s, \omega)$ with parabolic data $\sum_i \alpha_i p_i$ and $\text{Gal}(f)$--invariant vortex solutions on $(f^\ast L , f^\ast s , f^\ast \omega)$ with parabolic data $\sum_i (r_{f^{-1}(p_i)} +1 )\, \alpha_i \, f^{-1}(p_i)$, where $r_x$ is the ramification index of $f$ at the point $x$. This correspondence induces the holomorphic embedding of moduli spaces
\begin{align}\label{cover}
\widetilde{f}\, :\, \text{Sym}^d(X)\, &\longrightarrow\,
{\rm Sym}^{m.d} (Y) \\
\sum_{j=1}^d x_j \,&\longmapsto\,
\sum_{j=1}^d f^{-1}(x_j) \nonumber
\end{align}
where $f^{-1}(x_j)$ is the inverse image counted 
with multiplicities. The image of $\widetilde{f}$ is the subvariety of
Galois-invariant solutions inside the whole moduli space of solutions.

\begin{lemma}\label{lemma4}
Suppose that the metric on the surface $X$ has singularities with negative weight $\beta> -1$. Then the K\"ahler form of the $L^2$ metric $\omega_{\text{\rm s-vort}}^d$ on the moduli space ${\rm Sym}^d(X)$ is well-defined and satisfies the identity
\begin{equation}\label{2}
\omega_{\text{\rm s-vort}}^d (X)\ =\ \frac{1}{m}
\ \widetilde{f}^\ast \,\omega_{\text{\rm s-vort}}^{m.d} (Y)\, ,
\end{equation}
where $\widetilde{f}$ is the map in \eqref{cover} and $\omega_{\text{\rm s-vort}}^{m.d} (Y)$ is the K\"ahler form of the $L^2$-metric on the moduli space of $m.d$ vortices on the surface $Y$ equipped with the metric $f^\ast \omega$.
\end{lemma}

This lemma will be proved below. Its importance relies on the fact that the singularities of the K\"ahler form $f^\ast \omega$ on the cover $Y$ all have non-negative weight, as observed in the proof of Theorem \ref{existence}. So Lemma \ref{lemma2} is applicable, and using the parabolic weights at the inverse images $f^{-1}(p_i)$ specified in the proof of Theorem \ref{existence}, we obtain that
\[
[ \omega_{\text{\rm s-vort}}^{md} (Y)] \ = \ \pi\ \Big(\tau \, {\rm Vol}\, Y - 4 \pi\, m \big( d + \sum_i \alpha_i \big)\, \Big) \, \eta' \ + \ 4 \pi^2 \, \theta' \, ,
\]
where the classes $\eta'$ and $\theta'$ in $H^2 ({\rm 
Sym}^{md}(Y),\, {\mathbb R})$ are defined in the same way as $\eta\, ,\theta\,\in\,
H^2({\rm Sym}^d(X),\, {\mathbb R})$ were defined above. Since the metric on $Y$ is the
pullback of the metric on
$X$, the total volumes are related by ${\rm Vol}\, Y = m\, {\rm Vol}\, X\,=\,
\#\text{Gal}(f)\cdot {\rm Vol}\, X$. So in view of \eqref{2} to calculate the cohomology class $[\omega_{\text{\rm s-vort}}^d (X)]$ we only need to compute the pullback classes $\widetilde{f}^\ast\eta'$ and $\widetilde{f}^\ast\theta'$.

For the first one, take a point $y'\, \in\, Y$ such that $f$ is
unramified at $y'$. Take $x'\, :=\, f(y')\, \in\, X$. Let $H_{x'}\, \subset\,
\text{Sym}^d(X)$ (respectively, $H_{y'}\, \subset\,{\rm Sym}^{m.d} (Y)$) be the
hypersurface defined by the image of $\text{Sym}^{d-1}(X)$ (respectively,
${\rm Sym}^{md-1} (Y)$) under the embedding $(x_1\, ,\cdots\, ,
x_{d-1})\, \longmapsto\, (x'\, , x_1\, ,\cdots\, , 
x_{d-1})$ (respectively, $(y_1\, ,\cdots\, , 
y_{md-1})\, \longmapsto\, (y'\, , y_1\, ,\cdots\, ,
y_{md-1})$). So
$H_{x'}$ and $H_{y'}$ represent the cohomology classes $\eta$ and $\eta'$ respectively.
We have 
\[
H_{y'}\bigcap \widetilde{f}(\text{Sym}^d(X))\,=\, H_{x'} \ . 
\]
From this
it follows immediately that
\begin{equation}\label{t1}
\widetilde{f}^\ast\eta'\, = \, \eta \ . 
\end{equation}
Thus Theorem \ref{prop3} is a consequence of Lemma \ref{lem1} below.

\begin{lemma}\label{lem1}
The pullback $\widetilde{f}^\ast\theta''\,\in\, H^2({\rm 
Sym}^d(X),\,{\mathbb R})$ 
coincides with $m \cdot \theta$.
\end{lemma}

\begin{proof}
Let
$$
\widehat{f}\, :\, \text{Pic}^d(X)\, \longrightarrow\,
\text{Pic}^{m \cdot d}(Y)
$$
be the morphism defined by $\zeta\, \longmapsto\, f^*\zeta$.
The following diagram is evidently commutative:
$$
\begin{matrix}
\text{Sym}^d(X) & \stackrel{\gamma^{}_{X,d}}{\longrightarrow} & 
\text{Pic}^d(X)\\
~\Big\downarrow \widetilde{f} && ~ \Big\downarrow \widehat{f}\\
{\rm Sym}^{m \cdot d} (Y)
& \stackrel{\gamma^{}_{Y,m \cdot d}}{\longrightarrow} & 
\text{Pic}^{m \cdot d}(Y)
\end{matrix}
$$
where $\gamma^{}_{Y,m \cdot d}$ is constructed just as
$\gamma^{}_{X,d}$ is constructed in \eqref{gxd}. Therefore, to 
prove the lemma, it suffices to show that
\begin{equation}\label{aa}
\widehat{f}^*\theta_{\text{Pic}^{m \cdot d}(Y)}\,=\,
m \cdot \theta_{\text{Pic}^d(X)}\, .
\end{equation}

For any $\alpha\, ,\beta\,\in\, H^1(X,\, {\mathbb R})$, we have
\begin{equation}\label{eq-l}
\int_Y (f^\ast\alpha)\wedge (f^\ast\beta) \,=\,
m \int_X \alpha\wedge \beta\, .
\end{equation}
As noted before, suing the cup product on $H^1(X,\, {\mathbb R})$ (respectively,
$H^1(Y,\, {\mathbb R})$), the homology
$H_1(X,\, {\mathbb R})$ (respectively, $H_1(Y,\, {\mathbb R})$) gets identified
with $H^1(X,\, {\mathbb R})$ (respectively,
$H^1(Y,\, {\mathbb R})$). From \eqref{eq-l} it follows that the composition
$$
H^1(X,\, {\mathbb R})\, \stackrel{f^*}{\longrightarrow}\,
H^1(Y,\, {\mathbb R})\, =\, H_1(Y,\, {\mathbb R}) \, \stackrel{f_*}{\longrightarrow}\,
H_1(X,\, {\mathbb R})\, =\,H^1(X,\, {\mathbb R})
$$
coincides with multiplication by $m$. From this the equality
in \eqref{aa} follows.
\end{proof}

\subsection*{Proof of Lemma \ref{lemma4}}

Fix the pulled back K\"ahler form $\omega_Y$
on $Y$. Take any positive integer $m$. 
Let $$M^{m}_{\rm vort}\,=\, {\rm Sym}^{m} (Y)$$ be the moduli space of 
abelian pairs on $Y$ of degree $m$. This moduli space 
${\rm Sym}^m (Y)$ is equipped with a K\"ahler form $\omega_{\rm 
vort}^m (Y)$ constructed using $\omega_Y$; see \cite{IM}, \cite[p. 840, Definition 6.1]{BS}, \cite[p. 309, 
(4)]{Ba}. Below, we will
briefly recall the construction of $\omega_{\rm vort}^m (Y)$.

Take an abelian pair $\underline{y}\, :=\, (\widehat{L}\, 
,\sigma)\,\in\, \text{Sym}^m(Y)$. The hermitian structure on 
$\widehat{L}$ that satisfies the vortex equation for $\underline{y}$ with
respect to $\omega_Y$ will be denoted by
$h_{\underline{y}}$. The Chern connection on $\widehat{L}$ 
for $h_{\underline{y}}$ will be denoted by $\nabla^{\underline{y}}$. 
Let
\begin{equation}\label{vl}
v\, \in\, T_{\underline{y}} \text{Sym}^m(Y)
\end{equation}
be the
holomorphic tangent vector given by a holomorphic family of 
abelian pairs $(\widehat{L}_t\, ,\sigma_t)$, $t\, \in\, 
D_\epsilon$, with $(\widehat{L}_0\, ,\sigma_0)\,=\, 
(\widehat{L}\, ,\sigma)$. Let 
$h_t$ be the hermitian structure on $\widehat{L}_t$ that
satisfies the vortex equation for $(\widehat{L}_t\, ,\sigma_t)$
with respect to $\omega_Y$. Let $\widehat{\nabla}$ be the Chern
connection on $\widehat{L}$ for the family $(\widehat{L}_t\, ,h_t)$
of holomorphic hermitian line bundles.
For any $y\in\, Y$, taking
parallel translations, with respect to $\widehat{\nabla}$, along the radii
of the disk $\{y\}\times D_\epsilon$, we get a $C^\infty$
trivialization of the family of hermitian line bundles 
$\widehat{L}_t$. The Chern connection on $\widehat{L}_t$
for $h_t$ will be denoted by $\nabla^t$. Using the above
$C^\infty$ trivialization of the family of holomorphic line bundles
$\widehat{L}_t$, we have
$\nabla^t\,=\, \nabla^{\underline{y}}+\theta_t$
and $\sigma_t\,=\, \sigma+ \eta_t$, where $\theta_t$ is a 
$1$--form on $Y$
and $\eta_t$ is a smooth section of $\widehat{L}_0\,=\, 
\widehat{L}$. The K\"ahler form $\omega_{\rm vort}^m (Y)$ on 
$\text{Sym}^{m}(Y)$ is given by the formula
$$
\Vert v\Vert^2 \, :=\, \frac{1}{2} \int_Y 
(\frac{d\theta_t}{dt}(0))\wedge \star (\frac{d\theta_t}{dt}
(0)) + \int_Y \Vert \frac{d\eta}{dt}(0)\Vert^2\cdot\Omega_Y\, ,
$$
where $v$ is the tangent vector in \eqref{vl}. The lemma follows from
this description of K\"ahler form on $\text{Sym}^{m}(Y)$.

\section{Abelian $n$-pairs}

\subsection{Existence of solutions and volume of the moduli space}

In this section we will consider the vortex equations for 
abelian $n$-pairs. These are sometimes called semilocal vortices. 
An abelian $n$-pair is a holomorphic line bundle $L 
\longrightarrow X$ together with $n$ holomorphic sections $\phi_1, \ldots , 
\phi_n$, and we will assume that at least one of these sections is
nonzero.
The vortex equation for a hermitian metric $h$ on $L$ is then
\begin{equation} \label{n-vort-equation}
F_h - \sqrt{-1} e^2 \big(\vert \phi_1\vert^2_h + \cdots + \vert \phi_n\vert^2_h - \tau \big)\, \omega_X \,=\, 0\, .
\end{equation}
 Just as in the $n=1$ case treated before, 
if the 
line bundle $L$ is equipped with parabolic data we can consider the 
vortex equations with parabolic singularities. This means that 
we require $h$ to behave 
asymptotically like \eqref{par-assympt} at each 
singular point $p_i\,\in \,X$. We also allow finitely many degeneracies or conical singularities of the metric 
$\omega_X$ on the surface, so that it behaves like \eqref{oU} around each singular point $q_j \in X$. 
\begin{proposition}
Theorem \ref{existence} is valid for the vortex equation \eqref{n-vort-equation} with singularities.
\end{proposition}
\begin{proof}
The proof of Theorem \ref{existence} applies unchanged.
\end{proof}
Now call $M^{d, n}_{\text{\rm s-vort}}$ the moduli space of 
$n$-vortices with singularities. Just as for $n=1$, it follows from the 
proposition above that $M^{d, n}_{\text{\rm s-vort}}$ is isomorphic the space of 
abelian $n$-pairs on $L \longrightarrow X$ and that, as a complex manifold, it 
does not depend on the prescription of singularities. Thus it is the same as in the case of smooth $n$-vortices. In particular, for degree $d > 2 g_X - 2$, the manifold $M^{d, n}_{\text{\rm s-vort}}$ is a projective bundle over 
the Jacobian $\text{Pic}^d (X)$ with fiber $\mathbb{C}{\mathbb P}^{n(d+1 -g ) -1}$.
The natural 
$L^2$-metric on the moduli space, however, does depend on the existence of singularities. In \cite{Ba} it was shown that in the smooth case the K\"ahler class of the $L^2$-metric is
\[
[\omega^{d, n}_{\text{\rm vort}}] \ = \ \pi \Big( \tau \, {\rm Vol}\, X - 
\frac{2 \pi}{e^2}\, d \Big) \, \eta \ + \ \frac{2\pi^2}{e^2} \, \theta \, ,
\] 
where $\eta$ is the generator of the cohomology of the projective 
fiber and $\theta$ is the pull-back to $M^{d, n}_{\text{\rm 
vort}}$ of the theta class in $\text{Pic}^d (X)$. (In the $n=1$ 
case, these classes $\eta$ and $\theta$ coincide with the 
classes with the same name used earlier.) For 
semilocal vortices with singularities the analogous 
result is the following:

\begin{proposition}
Assume that the degree of the line bundle satisfies $\tau e^2 \, ({\rm Vol}\, X) / (2\pi)
- \sum_i \alpha_i > d > 2g_X -2$. Then the K\"ahler class of the natural $L^2$-metric on 
the moduli space $M^{d, n}_{\text{\rm s-vort}}$ of singular vortices is
\[
[\omega^{d, n}_{\text{\rm s-vort}}] \ = \ \pi \Big( \tau \, 
{\rm Vol}\, X - \frac{2\pi}{e^2}\, \big(d + \sum_{i=1}^r 
\alpha_i \big) \Big) \, \eta \ + \ \frac{2 \pi^2}{e^2} \, \theta \, .
\] 
\end{proposition}

\begin{proof}
Consider the natural projection $\pi_0 \,:\, M^{d, n}_{\text{\rm 
s-vort}} \,\longrightarrow\, 
\text{Pic}^d (X)$ that takes the class of a $n$-pair $(L, \phi_1 , \ldots , \phi_n)$ to the class of the line bundle $L$ in the Picard group. Since $d > 2g_X -2$ this projection defines a projective bundle, and we can write the K\"ahler class of the $L^2$ metric on the moduli space as
\[
[\omega^{d, n}_{\text{\rm s-vort}}] \ = \ \lambda\, \eta \ + 
\ \pi^\ast_0 \, \zeta \, ,
\]
where $\lambda$ is a scalar and $\zeta$ is a cohomology 2-class on the base
$\text{Pic}^d (X)$. Now let 
$$\iota : M^{d, 1}_{\text{\rm 
s-vort}} \longrightarrow M^{d, n}_{\text{\rm s-vort}}$$
 be the 
isometric embedding that takes $(L, h, \phi)$ to $(L, h, \phi, 0, 
\cdots, 0)$. The image of $\iota$ is a sub-bundle of $M^{d, 
n}_{\text{\rm s-vort}} \longrightarrow \text{Pic}^d (X)$ with 
fiber $\mathbb{C}\mathbb{P}^{d -g}$, so clearly $\iota^\ast \eta 
= \eta$ and $\iota^\ast \pi^\ast_0 \zeta\,=\,\pi^\ast_1 \zeta$, where
$\pi_1$ is the restriction of $\pi_0$ to the sub-bundle.
In other words,
\[
[\omega^{d, 1}_{\text{\rm s-vort}}] \ = \ \iota^\ast 
[\omega^{d, n}_{\text{\rm s-vort}}] \ = \ \lambda \, \eta \ + 
\ \pi^\ast_1 \, \zeta \, . 
\]
Comparing with \eqref{1-v-Kclass} we conclude that $\lambda$ and 
$\zeta$ are as stated in the proposition. This completes the 
proof.
\end{proof}
Since the K\"ahler class is known and $M^{d, n}_{\text{\rm 
s-vort}}$ is a projective bundle, the intersection calculations to compute the volume and total scalar curvature are quite simple \cite{Ba}. For example we have
\[
{\rm Vol}\; M^{d, n}_{\text{\rm s-vort}} \ = \ \pi^{l}\ \sum_{i=0}^g \ \frac{ g!\: n^{g-i}}{i! \: (l - i)! \: (g-i)!} \ \Big(\frac{2\pi}{e^2}\Big)^i \ \Big( \tau \, \big( {\rm Vol}\, X - \sum_{i} \frac{2 \pi}{e^2 \tau} \alpha_i \big) - \frac{2 \pi}{e^2}\, d \Big)^{l-i} \, ,
\]
where $l\,=\, g + n(d+1 - g) -1$. Both the 
volume and the total scalar curvature of the moduli space of 
vortices with singularities can be obtained from the corresponding values for 
the moduli space of vortices without singularities by the usual substitution ${\rm 
Vol}\, X \; \longrightarrow \; {\rm Vol}\, X - \sum_{i} \frac{2 
\pi}{e^2 \tau}\, \alpha_i \, $.

\subsection{Volume of spaces of holomorphic curves}

For $d \geq 2g$ the manifold $M^{d, n}_{\text{\rm s-vort}}$ can be regarded as a 
compactification of the space $\mathcal{H}^{d, n-1}$ of 
holomorphic maps $X \longrightarrow \mathbb{C}{\mathbb P}^{n-1}$ 
of degree 
$d$ (see \cite{BDW, Ba}). A heuristic argument similar to the one 
given in the last reference for regular vortices suggests that 
the pointwise limit as $e^2 \rightarrow \infty$ of the form 
$\omega^{d, n}_{\text{\rm s-vort}}$ over the domain 
$\mathcal{H}^{d, n-1} \,\subset\, M^{d, n}_{\text{\rm s-vort}}$ 
coincides with the K\"ahler form $\omega^{d, n-1}_{\text{\rm 
map}}$ of the natural $L^2$-metric on the space of maps $(X, 
\omega_X) \longrightarrow (\CC {\mathbb P}^{n-1}, \pi \tau 
\cdot \omega_{\text{\rm FS}})$, where $\omega_{\text{\rm FS}}$ is 
the 
normalized Fubini-Study form (its cohomology class is the generator
of the integral cohomology). Then the limit
\[
\lim_{e^2 \rightarrow \infty} {\rm Vol}\; M^{d, n}_{\text{\rm s-vort}} \ = \ \frac{  n^{g}}{\big[n(d+1-g) +g -1\big]! } \ \big( \pi \tau \, {\rm Vol}\, X \big)^{n(d+1 -g) +g -1}
\]
should presumably be interpreted as the volume of the non-compact 
Riemannian manifold $$(\mathcal{H}^{d, n-1} , \omega^{d, 
n-1}_{\text{\rm map}})\, .$$ Notice that the weights $\alpha_i$ disappear in this limit, and the result depends on $\omega_X$ only through the volume ${\rm Vol}\, X $.
This suggests that the conjectural formulae in \cite{Ba} for the volume and total scalar curvature of $\mathcal{H}^{d, n-1}$ should hold unchanged when $\omega_X$ has isolated singularities of the form \eqref{oU}.

\section{Vortices on hyperbolic surfaces}

Let $f\,:\, Y\,\longrightarrow\, X$ be a holomorphic map between two 
connected Riemann surfaces. The derivative $\dd f$ is a section of 
the holomorphic 
line bundle $L\, :=\, T^\ast Y \otimes f^\ast TX$. This bundle 
has a natural hermitian metric $h$ induced by the Riemannian metrics on 
the surfaces. The curvature form of this hermitian metric is 
\begin{equation}
F_h \ = \ -F_Y \ + \ f^\ast F_X \ = \ \sqrt{-1} (K_Y \ - \ |\dd f|^2 \ 
K_X) \ {\rm vol}_Y
\end{equation}
as a $2$-form over $Y$, where $K_X$ and $K_Y$ are the scalar curvatures of 
the surfaces. So if we take $X$ and $Y$ to be hyperbolic surfaces, i.e. 
if we take the scalar curvatures to be negative and constant, we see that $(L, \dd f , h )$ satisfies the abelian vortex equations over $Y$.
This curious fact was first observed by Witten \cite{Wi} in the case where $X$ and $Y$ are the hyperbolic plane; an observation that allowed him to explicitly construct all the vortex solutions on ${\mathbb H}^2$. Recently, after phrasing the observation in a more geometric language, Manton and Rink \cite{MR} have used it to study vortices on many other hyperbolic surfaces.
An extension of this result is the following.

\begin{proposition}\label{prop4}
Let $f\,:\, Y\,\longrightarrow\, X$ be a holomorphic map between 
hyperbolic surfaces of scalar curvature $K_Y =K_X = - \tau/2$. 
If $(L'\, 
, h') \,\longrightarrow\, X$ is a hermitian line bundle with a section 
$\phi'$ that satisfies the vortex equations over $X$, then the triple
$(L\, ,h\, ,\phi)$, where
\begin{enumerate}
\item $L \ := \ T^\ast Y \otimes f^\ast TX \otimes f^\ast L'$

\item $h \ := \ g^\ast_Y \cdot f^\ast g_X \cdot f^\ast h'$

\item $\phi \ := \ \dd f \otimes f^\ast \phi'$\, ,
\end{enumerate}
satisfies the vortex equations over $Y$. 
\end{proposition}

\begin{remark}
{\rm This means that for hyperbolic surfaces there is a natural way of 
obtaining vortex solutions by pull-back of other solutions. The 
observation of \cite{Wi} and \cite{MR} corresponds to the case 
where $L'$ is trivial, $h'=1$ and $\phi' = \sqrt{\tau}$.} Observe that the section $\phi$ vanishes exactly at the ramification points of $f$ plus the inverse image of the zeroes of $\phi'$.
\end{remark}

\begin{proof}
Using the vortex equation on $X$, the curvature of $(L, h)$ is
\begin{align*}
F_h \ &= \ -F_Y \ + \ f^\ast ( F_X + F_{h'}) \ = \ \sqrt{-1}\, K_Y \, {\rm 
vol}_Y 
\ - \ \sqrt{-1} ( K_X - \frac{1}{2} | f^\ast \phi' |_{h'}^2 + \frac{1}{2} 
\tau) \ 
f^\ast {\rm vol}_X \\
&= \ \frac{\sqrt{-1}}{2} ( -\tau + |\dd f|^2 | f^\ast \phi' |_{h'}^2 ) \ 
{\rm vol}_Y \ = \ \frac{\sqrt{-1}}{2} \big(|\phi|_h^2 - \tau \big) \ {\rm vol}_Y 
\ .
\end{align*}
So the vortex equation on $Y$ is satisfied.
\end{proof}

\begin{remark}
{\rm A similar calculation shows that the result also holds when 
$L'$ is a vector bundle over $X$. In this case, if $(L' \, , 
h'\, , \phi')$ satisfies the non-abelian vortex equations, so 
will $(L\, , h\, , \phi)$. A version of it holds for the 
coupled vortex equations as well.}
\end{remark}

\noindent
When the surfaces $X$ and $Y$ are compact the utility of Proposition \ref{prop4} as tool to obtain smooth vortex solutions on $Y$ is limited, as pointed out in \cite{MR}. This is because smooth hyperbolic metrics only exist for surfaces with genus 2 or more, by Gauss-Bonnet, and in this case the few holomorphic maps $f : Y \longrightarrow X$ that exist are isolated, i.e. do not have moduli. 
But notice that Proposition \ref{prop4} is equally 
valid when the metrics $g_X$, $g_Y$ and $h'$ have point singularities. In this case, 
of course, $h$ will in general be singular as well. It follows from the definition of $h$ that if $g_Y$ has a conical singularity of 
order $\beta_Y$ at a point $z=z_0$ and the metrics $g_X$ and $h'$ have a singularities of 
order $\beta_X$ and $\alpha'$, respectively, at $f(z_0)$, then the metric $h$ on $L$ has a parabolic point of 
weight $\alpha = \beta_X - \beta_Y + \alpha'$ at the point $z_0$.
So if are willing to consider vortex solutions with singularities, Proposition \ref{prop4} allows us to construct much bigger families of solutions. For example punctured Riemann spheres and punctured tori admit hyperbolic metrics, and moreover any meromorphic function on $Y$ defines a holomorphic map $Y \longrightarrow \CC \mathbb{P}^1$. 

As for explicit vortex solutions, a large family of them can be obtained if we take the surfaces $X$ and $Y$ to be the thrice punctured Riemann sphere and the map $f$ to be any rational map $\CC \mathbb{P}^1 \longrightarrow \CC \mathbb{P}^1$. In this case the unique hyperbolic metric on $\CC \mathbb{P}^1 \setminus \{0 , 1, \infty \}$ with singularities of order 
\begin{equation} \label{weights}
-1 < \beta_0 < 0 \, , \qquad -1 < \beta_1 < 0 \, , \qquad -1 < \beta_\infty < 0 \, , \qquad \beta_0 + \beta_1 + \beta_\infty < -2 
\end{equation}
has been written down explicitly in recent work of Kraus, Roth and Sugawa \cite{KRS}. Since rational maps on the sphere are also explicit, we can write very non-trivial singular vortex solutions on the punctured sphere. Although perfectly valid, most of these solutions will not fall within the class that has been studied in this article, because they will have points with negative parabolic weight $\alpha = \beta_X - \beta_Y + \alpha'$. In other words, the hermitian metric $h$ may explode at points in the inverse image $f^{-1}(\{0 , 1, \infty \})$. 
Nevertheless, we can still write a simple example where this explicit construction gives a continuous and non-trivial vortex solution $h$. Just take $f$ to be the identity map; $L'$ to be the trivial bundle with 
$h'=1$ and $\phi' = \sqrt{\tau}$; and choose the hyperbolic metrics $g_Y$ and $g_X$ on the punctured sphere $\CC \mathbb{P}^1 \setminus \{0 , 1, \infty \}$ with negative weights at the three singularities satisfying \eqref{weights} and 
\[
(\beta_Y)_0 \leq (\beta_X)_0 \ , \qquad (\beta_Y)_1 \leq (\beta_X)_1 \ , \qquad (\beta_Y)_\infty \leq (\beta_X)_\infty \ .
\]
Then $L$ is the trivial bundle, the section $\phi$ is constant and equal to $\sqrt{\tau}$, and the hermitian metric
\begin{equation} \label{explicitsolution}
h \ =\  g_{X,\, \beta_X} \ /\  g_{Y,\, \beta_Y} \ 
 \end{equation}
 is a non-trivial vortex solution on $L \longrightarrow (Y, g_Y)$ with parabolic weight $(\beta_X)_p - (\beta_Y)_p$ at each of the three singular points $p= 0, 1, \infty$. Using the formulae of the Appendix, this vortex solution $h$ can be written down in terms of hypergeometric functions on $\CC \setminus \{ 0, 1\}$. The relation between vortex solutions and hyperbolic metrics illustrated in this example will be clarified in future work.

\medskip
\noindent
\textbf{Acknowledgements.}\, The second author wishes to thank
Instituto Superior T\'ecnico, where a large part of the work was carried out,
for its hospitality. The visit to IST was funded by the FCT project
PTDC/MAT/099275/2008. The first author thanks CAMGSD and Project  PTDC/MAT/119689/2010 of FCT - POPH/FSE for a generous fellowship.

\section*{Appendix}

For the sake of completeness we will reproduce here, using our notation, the explicit hyperbolic metrics on $\CC \mathbb{P}^1 \setminus \{ 0, 1 , \infty \}$ of curvature $-1$ found by Kraus, Roth and Sugawa \cite{KRS}. By \eqref{explicitsolution} they determine explicit and non-trivial vortex solutions. Identify the thrice punctured sphere with the punctured plane $\CC \setminus \{ 0, 1\}$ and choose weights on the singularities satisfying $-1 < \beta_0 , \, \beta_1 \, , \beta_\infty < 0$ and $ \beta_0 + \beta_1 + \beta_\infty < -2$. Define the constants
\[
\lambda\ := \ - (\beta_0 + \beta_1 - \beta_\infty)/2 \, , \qquad \delta \ := \ - (\beta_0 + \beta_1 + \beta_\infty - 2)/2 \, , \qquad \gamma \ := - \beta_0 \ .
\] 
Then $0 < \delta \leq \lambda$ and $\lambda + \delta < \gamma < 1$. Consider the hypergeometric functions on the plane
\[
\varphi_1 (z) \ := \ F (\lambda , \delta , \gamma;\, z)\, , \qquad \varphi_2 (z) \ := \ F (\lambda , \delta , \lambda + \delta - \gamma +1;\, 1 - z) \ ,
\]
so that $\varphi_1$ is analytic on $\CC \setminus [1, +\infty)$ and $\varphi_2$ is analytic on $\CC \setminus ( - \infty , 0 ]$. Then the hyperbolic metric on the thrice punctured sphere with the chosen weights is given by
\[
g \ = \ g_{\beta_0 , \beta_1 , \beta_\infty} (z) \ \dd z \otimes \dd \bar{z} \ ,
\]
where the coefficient function is 
\[
g_{\beta_0 , \beta_1 , \beta_\infty} (z) \ = \ \frac{ 2\, \vert z \vert^{\beta_0}\, \vert 1- z \vert^{\beta_1}\, K_3 }{K_1 \, \vert \varphi_1 (z) \vert^2\, +\, K_2 \, \vert \varphi_2 (z) \vert^2 \, + \, 2 \, {\rm Re}(\varphi_1 (z) \, \varphi_2 (\bar{z}))} \ ,
\]
and the remaining constants are defined by
\begin{align*}
K_1\ &:= \ - \, \frac{\Gamma (\gamma - \lambda) \, \Gamma(\gamma - \delta)}{\Gamma (\gamma) \, \Gamma(\gamma - \lambda - \delta)} \, \qquad 
K_2\ := \ - \, \frac{\Gamma (\lambda +1 -\gamma) \, \Gamma(\delta +1- \gamma)}{\Gamma (1 - \gamma) \, \Gamma(\lambda + \delta + 1 - \gamma)} \, , \\
K_3\ &:= \ \sqrt{ \frac{\sin(\pi \lambda) \, \sin(\pi \delta)}{\sin(\pi (\gamma - \lambda))\, \sin(\pi(\gamma - \delta))} } \ \cdot \ 
\frac{\Gamma (\lambda + \delta + 1 - \gamma) \, \Gamma(\gamma)}{\Gamma(\lambda) \, \Gamma(\delta)} \ .
\end{align*}


\end{document}